\documentclass{conm-p-l}

\usepackage{amssymb}

\newtheorem{theorem}{Theorem}[section]

\theoremstyle{definition}
\newtheorem{definition}[theorem]{Definition}

\theoremstyle{remark}

\numberwithin{equation}{section}

\newcommand{\doesnotdivide}{\nmid}
\newcommand{\partition}[1]{\ensuremath{\left\langle{#1}\right\rangle}}
\newcommand{\num}{\nu}
\newcommand{\GEQ}{\succcurlyeq}

\newcommand{\genqbinom}[3]{\ensuremath{\genfrac{[}{]}{0pt}{}{#1}{#2}_{#3}}}

\begin{document}

\title[An inequality of products of two $q$-Pochhammer symbols]{A partition inequality involving products of two $q$-Pochhammer symbols}

\author{Alexander Berkovich}
\address{Department of Mathematics, University of Florida, Gainesville, FL 32611-8105}
\email{alexb@ufl.edu}

\author{Keith Grizzell}
\address{Department of Mathematics, University of Florida, Gainesville, FL 32611-8105}
\email{grizzell@ufl.edu}

\date{\today.}

\dedicatory{This manuscript is dedicated to the memory of Srinivasa Ramanujan.}


\begin{abstract}
We use an injection method to prove a new class of partition inequalities involving certain   $q$-products with two to four finitization parameters. Our new theorems are a substantial  generalization of work by Andrews and of previous work by Berkovich and Grizzell. We also briefly discuss how our products might relate to lecture hall partitions.
\end{abstract}

\maketitle

\section{Introduction} \label{sec:introduction}

The celebrated Rogers-Ramanujan identities \cite{RR} are given analytically as follows:
\begin{equation}
\sum^{\infty}_{n=0}\frac{q^{n^2}}{(q;q)_n} = \frac{1}{(q,q^4;q^5)_{\infty}} \label{RR1}
\end{equation}
and
\begin{equation}
\sum^{\infty}_{n=0}\frac{q^{n^2+n}}{(q;q)_n} = \frac{1}{(q^2,q^3;q^5)_{\infty}}. \label{RR2}
\end{equation}
Here we are using the following standard notations:
\begin{align*}
(a;q)_L &= \begin{cases}
              1 & \text{if } L=0,\\
              \prod_{j=0}^{L-1}(1-aq^j) &\text{if } L>0,
           \end{cases}\\
(a_1,a_2,\dots,a_n;q)_L &= (a_1;q)_L (a_2;q)_L \cdots (a_n;q)_L,\\
(a;q)_\infty &= \lim_{L\to\infty} (a;q)_L.
\end{align*}
Subtracting \eqref{RR2} from \eqref{RR1} we have
\begin{equation}
\sum^{\infty}_{n=1}\frac{q^{n^2}}{(q;q)_{n-1}}=
\frac{1}{(q,q^4;q^5)_{\infty}}- 
\frac{1}{(q^2,q^3;q^5)_{\infty}}, 
\label{RRdiff}
\end{equation}
from which it is obvious that the coefficients in the $q$-series expansion of the difference of the two products in \eqref{RRdiff} are all non-negative. 
In other words, for all $n>0$ we have
\begin{equation}
p_{1}(n) \geq p_{2}(n), \label{RRdiffThm}
\end{equation}
where $p_{r}(n)$ denotes the number of partitions of $n$ into parts congruent to $\pm r \pmod{5}$.

At the 1987 A.M.S.\ Institute on Theta Functions, Leon Ehrenpreis asked if one can prove \eqref{RRdiffThm} without resorting to the Rogers-Ramanujan identities.
In 1999, Kevin Kadell \cite{KK} provided an affirmative answer to this question by constructing an injection of partitions counted by $p_{2}(n)$ into partitions counted by $p_{1}(n)$.
In 2005, Alexander Berkovich and Frank Garvan \cite{BG} constructed an injective proof for an infinite family of partition function inequalities related to finite products, thus giving us the following theorem.
\begin{theorem}
\label{BGThm}
Suppose $L>0$, and $1<r<m-1$. Then the
coefficients in the $q$-series expansion of the difference of the two finite products
\begin{equation*}
\frac{1}{(q,q^{m-1};q^m)_L}-
\frac{1}{(q^r,q^{m-r};q^m)_L} 
\end{equation*}
are all non-negative if and only if $r \doesnotdivide (m-r)$ and $(m-r) \doesnotdivide r$.
\end{theorem}
\noindent
We note that \eqref{RRdiffThm} is an immediate corollary of this theorem with $m=5$, $r=2$ and $L\rightarrow\infty$.

In 2012, drawing inspiration from George Andrews, Alexander Berkovich and Keith Grizzell proved the following theorem in \cite{BGr}. (Andrews had used his anti-telescoping technique to prove the $y=3$ case of the following theorem in \cite{An1}.)
\begin{theorem}
\label{BerkGrizThm}
For any $L>0$ and any odd $y>1$, the $q$-series expansion of
\begin{equation}
\frac{1}{(q,q^{y+2},q^{2y};q^{2y+2})_L} - \frac{1}{(q^2,q^{y},q^{2y+1} ;q^{2y+2})_L} = \sum_{n=1}^\infty a(L,y,n) q^{n} \label{BerkGrizEqn}
\end{equation}
has non-negative coefficients.
\end{theorem}
We note that the products on the left of \eqref{BerkGrizEqn} can be interpreted as 
\begin{align}
\frac{1}{(q,q^{y+2},q^{2y} ;q^{2y+2})_L} &= 1 + \sum_{n=1}^\infty P_{1}'(L,y,n) q^{n}\\
\intertext{and}
\frac{1}{(q^2,q^{y},q^{2y+1} ;q^{2y+2})_L} &= 1 + \sum_{n=1}^\infty P_{2}'(L,y,n) q^{n},
\end{align}
where  $P_1'(L,y,n)$ denotes the number of partitions of $n$ into parts $\equiv 1,y+2,2y \pmod{2(y+1)}$ with the largest part not exceeding $2(y+1)L-2$ and $P_2'(L,y,n)$ denotes the number of partitions of $n$ into parts $\equiv 2,y,2y+1 \pmod{2(y+1)}$ with the largest part not exceeding $2(y+1)L-1$.

These problems all belong to a broad class of positivity problems in $q$-series and partitions which often are very deceptive because they are so easy to state but so painfully hard to solve. For example, consider the famous problem from Peter Borwein:
\begin{quote}
Let $B_e(L,n)$ (resp.\ $B_o(L,n)$)  denote the number of partitions of $n$ into an even (resp.\ odd) number  of distinct nonmultiples of $3$ with each part less than $3L$. Prove that for all positive integers $L$ and $n$,
$B_e(L,n)-B_o(L,n)$  is nonnegative if $n$ is a multiple of $3$ and nonpositive otherwise.
\end{quote}
As of the date of this manuscript, this conjecture still remains unproved despite the efforts of many excellent mathematicians. (For further background on this conjecture we refer the reader to \cite{An4}, \cite{BW}, \cite{BR}, \cite{W1}, and \cite{W2}.)

There is a useful notation that can be used to succinctly convey the fact that coefficients of a difference of two $q$-series are nonnegative.
\begin{definition}
  Let $A(q):=\sum_{x\geq0}a_xq^x$ and $B(q):=\sum_{x\geq0}b_xq^x$ be two $q$-series. Then $A(q) \GEQ B(q)$ if and only if $a_x \geq b_x$ for all $x\geq0$.
\end{definition}

Clearly we could multiply or divide every exponent of $q$ in any inequality $A(q) \GEQ B(q)$ by some common factor or divisor to trivially obtain an equally valid inequality. So, if the exponents share no common integer factor greater than 1, we consider the inequality to be irreducible. 
So far, when examining an irreducible inequality of the form
\begin{equation}\label{GenericQproductIneq}
\frac{1}{\prod_{r=1}^{s}(1-q^{a_r})} \GEQ \frac{1}{\prod_{r=1}^{s}(1-q^{b_r})},
\end{equation}
where $a_1 \leq \cdots \leq a_s$ and $b_1 \leq \cdots \leq b_s$, it has been the case that $a_1=1$.
In 2012, at the Ramanujan 125 Conference in Gainesville, Florida, Hamza Yesilyurt asked if the inclusion of the factor $(1-q)$ was necessary in all such irreducible inequalities. At the time, no one proffered a definitive answer, though we generally agreed that that was our experience, in that the irreducible inequalities we had seen at that point all included the factor $(1-q)$. Shortly after the conference, however, Berkovich and Grizzell proved, among other things in \cite{BGr2}, the following theorem which indicates that the inclusion of the factor $(1-q)$ is not necessary.
\begin{theorem}\label{xyzthm}
For any octuple of positive integers $(L,m,x,y,z,r,R,\rho)$, 
\begin{equation*}
\frac{1}{(q^x,q^y,q^z,q^{rx+Ry+\rho z};q^m)_L}\GEQ
\frac{1}{(q^{rx},q^{Ry},q^{\rho z},q^{x+y+z};q^m)_L}.
\end{equation*}
\end{theorem}

The main object of the present manuscript is the following new theorem which significantly generalizes Theorem \ref{BerkGrizThm} and at the same time provides another source of nontrivial $q$-product inequalities of the form \eqref{GenericQproductIneq} with $a_1 > 1$.
\begin{theorem} \label{MainThm}
For any positive integers $m$, $n$, $y$, and $z$, with $\gcd(n,y)=1$, and integers $K$ and $L$, with $K \geq L \geq 0$, 
\begin{equation}
\frac{1}{(q^z;q^{m})_K (q^{nyz};q^{nm})_L} \GEQ \frac{1}{(q^{yz};q^{m})_K (q^{nz};q^{nm})_L}. \label{MainIneq}
\end{equation}
\end{theorem}
By taking $z=1$, $n=2$, $m=y+1$, and $K=2L$, \eqref{MainIneq} becomes
\begin{equation*}
\frac{1}{(q,q^{y+2},q^{2y};q^{2y+2})_L} \GEQ \frac{1}{(q^2,q^{y},q^{2y+1} ;q^{2y+2})_L},
\end{equation*}
which yields Theorem \ref{BerkGrizThm}.

We note that the products in \eqref{MainIneq} can be interpreted as 
\begin{align}
\frac{1}{(q^z;q^{m})_K (q^{nyz};q^{nm})_L} &= 1 + \sum_{x=1}^\infty P_{1}(x,K,L,m,n,y,z) q^{x} \label{firstproduct}\\
\intertext{and}
\frac{1}{(q^{yz};q^{m})_K (q^{nz};q^{nm})_L} &= 1 + \sum_{x=1}^\infty P_{2}(x,K,L,m,n,y,z) q^{x}, \label{secondproduct}
\end{align}
where  $P_{1}(x,K,L,m,n,y,z)$ denotes the number of partitions of $x$ into parts congruent to $z$ modulo $m$, with the largest part not exceeding $(K-1)m+z$, and parts congruent to $nyz$ modulo $nm$, with the largest part not exceeding $(L-1)nm+nyz$; and $P_{2}(x,K,L,m,n,y,z)$ denotes the number of partitions of $x$ into parts congruent to $yz$ modulo $m$, with the largest part not exceeding $(K-1)m+yz$, and parts congruent to $nz$ modulo $nm$, with the largest part not exceeding $(L-1)nm+nz$. Also, it is possible that the same part could arise in multiple ways; in this case, we may simply introduce another distinguishing feature, such as assigning different colors to the parts with different origins, in order to tell them apart.

In the next section, we define some notation that we will use to simplify the presentation of the proofs. The proof of Theorem \ref{MainThm} is accomplished by constructing an injection in section \ref{sec:injection}. In section \ref{sec:inverse} we provide the inverse map to the injection (thus supporting the claim that the map constructed in section \ref{sec:injection} is indeed an injection).
In section \ref{sec:examples}, we provide examples of the injection in action.

In section \ref{sec:generalization}, we prove the following dual to Theorem \ref{MainThm}, in which the $K$ and $L$ on the right-hand side of \eqref{MainIneq} are swapped.
\begin{theorem} \label{DualThm}
For any positive integers $m$, $n$, $y$, and $z$, with $\gcd(n,y)=1$, and integers $K$ and $L$, with $K \geq L \geq 0$, 
\begin{equation}
\frac{1}{(q^z;q^{m})_K (q^{nyz};q^{nm})_L} \GEQ \frac{1}{(q^{yz};q^{m})_L (q^{nz};q^{nm})_K}. \label{DualIneq}
\end{equation}
\end{theorem}
\noindent We then show how Theorems \ref{MainThm} and \ref{DualThm} can be generalized to the following.
\begin{theorem}\label{GenThm}
For any positive integers $m$, $n$, $y$, and $z$, with $\gcd(n,y)=1$; integers $K$ and $L$, with $K \geq L \geq 0$; and integers $S$ and $T$, with $\max(S,T)\leq K$ and $0\leq \min(S,T)\leq L$; 
\begin{equation}
\frac{1}{(q^z;q^{m})_K (q^{nyz};q^{nm})_L} \GEQ \frac{1}{(q^{yz};q^{m})_S (q^{nz};q^{nm})_T}. \label{GenIneq}
\end{equation}
\end{theorem}

In section \ref{sec:invariant}, we discuss two partition invariants that are preserved by the injections we present, as well as the implications of this invariance, namely Theorems
\ref{Refinement1} and \ref{Refinement2}, which may be regarded as refinements of Theorems \ref{MainThm} and \ref{DualThm}, respectively.
Finally, in section \ref{sec:conclusion}, we conclude with a brief discussion of how Theorems \ref{MainThm} and \ref{DualThm} might relate to lecture hall partitions.

\section{Notation}\label{sec:notation}

Let $n>1$ and $y>1$ be positive integers with $\gcd(n,y)=1$. Let $z$ and $m$ be positive integers.

\begin{definition}
Let 
\begin{align*}
  z_i &:= z + (i-1)m \\
\intertext{and}
  (nyz)_i &:= nyz + (i-1)nm.
\end{align*}
Then the product 
\begin{equation*}
  \frac{1}{(q^z;q^m)_K (q^{nyz};q^{nm})_L}
\end{equation*}
can be thought of as the generating function for the number of partitions into parts from the set $\{z_1, \dots, z_K, (nyz)_1, \dots, (nyz)_L\}$.
Similarly, let
\begin{align*}
  (yz)_i &:= yz + (i-1)m \\
\intertext{and}
  (nz)_i &:= nz + (i-1)nm .
\end{align*}
Then the product 
\begin{equation*}
  \frac{1}{(q^{yz};q^m)_K (q^{nz};q^{nm})_L}
\end{equation*}
can be thought of as the generating function for the number of partitions into parts from the set $\{(yz)_1, \dots, (yz)_K, (nz)_1, \dots, (nz)_L\}$.
\end{definition}
In the event of the same part occurring in more than one way, for example if $z_5 = (nyz)_2$, then by using the notation above we are implying that those parts can be distinguished from each other, as if they also had colors that were different. (Feel free to paint your own picture!)

\begin{definition}
Let $\num(p,\pi)$ denote the number of occurrences of the part $p$ in the partition $\pi$.
Let $Q(p,\pi)$ and $R(p,\pi)$ be the uniquely determined non-negative integers such that $\num(p,\pi)=n\cdot Q(p,\pi)+R(p,\pi)$ and $0 \leq R(p,\pi) \leq n-1$.
\end{definition}

\begin{definition}
Let \partition{1^{a_1},2^{a_2},3^{a_3},\dots,k^{a_k},\dots}, with $\sum a_k < \infty$, be the unique partition $\pi$ such that $\num(k,\pi)=a_k$ for every integer $k\geq1$.
\end{definition}

\begin{definition}
Let the norm of a partition $\pi=\partition{1^{a_1},2^{a_2},3^{a_3},\dots,k^{a_k},\dots}$, denoted $|\pi|$, be given by $|\pi|=\sum a_k k$.
\end{definition}

\section{The Injection}\label{sec:injection}

Suppose that $L=0$; then an injection mapping a partition $\pi_2$ (counted by $P_2(x,K,0,m,n,y,z)$), into a partition $\pi_1$ (counted by $P_1(x,K,0,m,n,y,z)$), where $|\pi_2|=|\pi_1|=x$, is given by
\begin{equation*}
\num(z_i,\pi_1) = 
  \begin{cases}
    \num((yz)_i,\pi_2) & \text{if } 1 < i \leq K,\\
    y \cdot \num((yz)_1,\pi_2) + (y-1)\cdot\displaystyle\sum_{1 < j \leq K} \num((yz)_j,\pi_2) &\text{if }i=1.
  \end{cases}
\end{equation*}

Suppose instead that $L>0$; then an injection mapping a partition $\pi_2$ (counted by $P_2(x,K,L,m,n,y,z)$), into a partition $\pi_1$ (counted by $P_1(x,K,L,m,n,y,z)$), where $|\pi_2|=|\pi_1|=x$, is given by
\begin{equation*}
\num((nyz)_i,\pi_1) = Q((yz)_i,\pi_2) \quad \text{(if $1\leq i \leq L$)}
\end{equation*}
and
\begin{equation*}
\num(z_i,\pi_1) = 
  \begin{cases}
    \num((yz)_i,\pi_2) & \text{if } L < i \leq K,\\
    n\cdot\num((nz)_i,\pi_2) + R((yz)_i,\pi_2) & \text{if } 1 < i \leq L,\\
    n\cdot \num((nz)_1,\pi_2) + y\cdot R((yz)_1,\pi_2) + (y-1)\cdot(A+B) &\text{if }i=1,
  \end{cases}
\end{equation*}
where
\begin{equation*}
A :=  \sum_{1 < j \leq L} R((yz)_j,\pi_2) \quad \text{ and } \quad B:= \sum_{L < j \leq K} \num((yz)_j,\pi_2).
\end{equation*}
We note that $x$ and $m$ appear nowhere in the statement of the injection, and that the injection is essentially independent of $z$ as well; contrast this with the fact that $n$, $y$, $K$, and $L$ all play very crucial roles in the injection.

Readers may find that it is helpful to consult the atlas in Table \ref{table:atlas} and the examples in section \ref{sec:examples} to better understand the injection.

\begin{table}[ht]
\caption{An atlas for the injection.}\label{table:atlas}
\begin{tabular}{c@{${}\mapsto{}$}cl}\hline
\vphantom{$\dfrac{1}{2}$}\partition{(nz)_i} & \partition{z_i^n} & for $1 \leq i \leq L$\\[0.5ex]
\partition{(yz)_1^{nk+j}} & \partition{z_1^{yj},(nyz)_1^k} & for $0 \leq j < n$, $L>0$\\[0.5ex]
\partition{(yz)_i^{nk+j}} & \partition{z_1^{(y-1)j},z_i^j,(nyz)_i^k} & for $1 < i \leq L$, $0 \leq j < n$\\[0.5ex]
\partition{(yz)_i} & \partition{z_1^{y-1},z_i} & for $L < i \leq K$\\[0.5ex]
\hline
\end{tabular}
\end{table}

\section{The Inverse}\label{sec:inverse}

Suppose that $L=0$; then the inverse mapping $\pi_1 \mapsto \pi_2$ is given by
\begin{equation*}
  \num((yz)_i,\pi_2) = 
  \begin{cases}
    \num(z_i,\pi_1) & \text{if } 1 < i \leq K,\\
    \num(z_1,\pi_1)-\dfrac{y-1}{y}\cdot\displaystyle\sum_{1 \leq j \leq K} \num(z_j,\pi_1) &\text{if }i=1.
  \end{cases}
\end{equation*}
Clearly in this case a partition $\pi_1$ is mapped to if and only if 
\begin{equation*}
  \num(z_1,\pi_1) \geq (y-1)\cdot\sum_{2 \leq j \leq K} \num(z_j,\pi_1).
\end{equation*}

When $L>0$, the inverse is a little bit more complicated. First, it is easy to see from the definition of the injection that $A$ and $B$ satisfy
\begin{equation*}
  A = \sum_{1 < j \leq L} R((yz)_j,\pi_2) = \sum_{1 < j \leq L} R(z_j,\pi_1) 
\end{equation*}
and
\begin{equation*}
  B = \sum_{L < j \leq K} \num((yz)_j,\pi_2) = \sum_{L < j \leq K} \num(z_j,\pi_1).
\end{equation*}
If we let $C(\pi_1)$ be the least nonnegative residue of $\overline{y}\cdot\mu(\pi_1)$ modulo $n$, let $\overline{y}$ be the multiplicative inverse of $y$ modulo $n$ (which is well-defined since $\gcd(n,y)=1$), take 
\begin{equation*}
A =  \sum_{1 < j \leq L} R(z_j,\pi_1) \quad \text{ and } \quad B= \sum_{L < j \leq K} \num(z_j,\pi_1),
\end{equation*}
and define
\begin{equation*}
  \mu(\pi_1) := \num(z_1,\pi_1) -(y-1)\cdot(A+B),
\end{equation*}
then for $L>0$ the inverse mapping $\pi_1 \mapsto \pi_2$ is given by
\begin{equation*}
\num((yz)_i,\pi_2) = 
  \begin{cases}
    \num(z_i,\pi_1) & \text{if } L < i \leq K,\\
    n\cdot\num((nyz)_i,\pi_1) + R(z_i,\pi_1) & \text{if } 1 < i \leq L,\\
    n\cdot \num((nyz)_1,\pi_1) + C(\pi_1) &\text{if }i=1,
  \end{cases}
\end{equation*}
and
\begin{equation*}
\num((nz)_i,\pi_2) =
  \begin{cases}
    Q(z_i,\pi_1) & \text{if } 1 < i \leq L,\\
    \dfrac{\mu(\pi_1)-y\cdot C(\pi_1)}{n} & \text{if } i=1.
  \end{cases}
\end{equation*}
In this case (when $L>0$) a partition $\pi_1$ gets mapped to if and only if $\mu(\pi_1)$ is a linear combination of the form $\mu(\pi_1)=a\cdot n + b \cdot y$, where $a\geq0$ and $0\leq b < n$. Note that if $\mu(\pi_1)$ is such a linear combination, then $a$ and $b$ are unique since $\gcd(n,y)=1$.

\section{Examples}\label{sec:examples}

If we take $z=1$, $n=2m-1$, $y=m-1$, and $K\geq 2m-2$, then we must, just by the nature of the arithmetic sequences involved, be dealing with duplicated parts, and thus require some means (e.g.\ \textbf{boldface} font) to distinguish them. So, for example, if we take $(K,L,m,n,y,z)=(4,2,3,5,2,1)$, then the domain partitions are constructed from the set of parts 
\[\{2_1,2_2,2_3,2_4,5_1,5_2\} = \{2,5,8,11,\mathbf{5},20\}\]
and the codomain partitions are constructed from the set of parts
\[\{1_1,1_2,1_3,1_4,10_1,10_2\} = \{1,4,7,10,\mathbf{10},25\}.\]
So, with $(K,L,m,n,y,z)=(4,2,3,5,2,1)$, if we consider, for example, partitions of $x=20$, then we have the complete injection shown in Table \ref{table:example1}.

\begin{table}[h]
\caption{The injection $\pi_2 \mapsto \pi_1$ for $(K,L,m,n,y,z)=(4,2,3,5,2,1)$, where $|\pi_2|=|\pi_1|=20$. Note that the last six lines of the table are not part of the injection, i.e.\ those $\pi_1$ have no pre-image $\pi_2$. Also, $5=2_2$, $\mathbf{5}=5_1$, $10=1_4$, and $\mathbf{10}=10_1$.}\label{table:example1}
\begin{tabular}{c@{\ $\mapsto$\ }cr@{$\ \ {}={}$}r@{$n + {}$}r@{$y$}}
\hline\vphantom{$\dfrac{1}{2}$}
$\pi_2$ & $\pi_1$ & $\mu(\pi_1)$ & $a$ & $b$\\[0.5ex]
\hline\vphantom{$\dfrac{1}{2}$}
\partition{20^{1}} & \partition{4^{5}} & $0$ & $0$ & $0$\\[0.5ex]
\partition{\mathbf{5}^{4}} & \partition{1^{20}} & $20$ & $4$ & $0$\\[0.5ex]
\partition{5^{1},\mathbf{5}^{3}} & \partition{1^{16},4^{1}} & $15$ & $3$ & $0$\\[0.5ex]
\partition{5^{2},\mathbf{5}^{2}} & \partition{1^{12},4^{2}} & $10$ & $2$ & $0$\\[0.5ex]
\partition{5^{3},\mathbf{5}^{1}} & \partition{1^{8},4^{3}} & $5$ & $1$ & $0$\\[0.5ex]
\partition{5^{4}} & \partition{1^{4},4^{4}} & $0$ & $0$ & $0$\\[0.5ex]
\partition{2^{1},8^{1},\mathbf{5}^{2}} & \partition{1^{13},7^{1}} & $12$ & $2$ & $1$\\[0.5ex]
\partition{2^{1},5^{1},8^{1},\mathbf{5}^{1}} & \partition{1^{9},4^{1},7^{1}} & $7$ & $1$ & $1$\\[0.5ex]
\partition{2^{1},5^{2},8^{1}} & \partition{1^{5},4^{2},7^{1}} & $2$ & $0$ & $1$\\[0.5ex]
\partition{2^{2},11^{1},\mathbf{5}^{1}} & \partition{1^{10},10^{1}} & $9$ & $1$ & $2$\\[0.5ex]
\partition{2^{2},8^{2}} & \partition{1^{6},7^{2}} & $4$ & $0$ & $2$\\[0.5ex]
\partition{2^{2},5^{1},11^{1}} & \partition{1^{6},4^{1},10^{1}} & $4$ & $0$ & $2$\\[0.5ex]
\partition{2^{5},\mathbf{5}^{2}} & \partition{1^{10},\mathbf{10}^{1}} & $10$ & $2$ & $0$\\[0.5ex]
\partition{2^{5},5^{1},\mathbf{5}^{1}} & \partition{1^{6},4^{1},\mathbf{10}^{1}} & $5$ & $1$ & $0$\\[0.5ex]
\partition{2^{5},5^{2}} & \partition{1^{2},4^{2},\mathbf{10}^{1}} & $0$ & $0$ & $0$\\[0.5ex]
\partition{2^{6},8^{1}} & \partition{1^{3},7^{1},\mathbf{10}^{1}} & $2$ & $0$ & $1$\\[0.5ex]
\partition{2^{10}} & \partition{\mathbf{10}^{2}} & $0$ & $0$ & $0$\\[0.5ex]
 & \partition{10^{1},\mathbf{10}^{1}} & $-1$ & $-1$ & $2$\\[0.5ex]
 & \partition{10^{2}} & $-2$ & $-2$ & $4$\\[0.5ex]
 & \partition{1^{1},4^{3},7^{1}} & $-3$ & $-1$ & $1$\\[0.5ex]
 & \partition{1^{2},4^{1},7^{2}} & $-1$ & $-1$ & $2$\\[0.5ex]
 & \partition{1^{2},4^{2},10^{1}} & $-1$ & $-1$ & $2$\\[0.5ex]
 & \partition{1^{3},7^{1},10^{1}} & $1$ & $-1$ & $3$\\[0.5ex]
\hline
\end{tabular}  
\end{table}

While it is nice to see an example with the actual numbers in it, as in Table \ref{table:example1}, it is easier to follow the injection patterns if the numbers are written using the subscript notation. Of course, this also has the added benefit of not requiring the use of some other means of distinguishing parts that happen to have the same value. So, we give the same slice of the injection, with $(K,L,m,n,y,z)=(4,2,3,5,2,1)$ and considering partitions of $x=20$, but instead written using subscript notation, in Table \ref{table:example1subscripted}.

\begin{table}[h]
\caption{Table \ref{table:example1} re-written in subscripted form.}\label{table:example1subscripted}
\begin{tabular}{c@{\ $\mapsto$\ }cr@{$\ \ {}={}$}r@{$n + {}$}r@{$y$}}
\hline\vphantom{$\dfrac{1}{2}$}
$\pi_2$ & $\pi_1$ & $\mu(\pi_1)$ & $a$ & $b$\\[0.5ex]
\hline\vphantom{$\dfrac{1}{2}$}
\partition{5_{2}^{1}} & \partition{1_{2}^{5}} & $0$ & $0$ & $0$\\[0.5ex]
\partition{5_{1}^{4}} & \partition{1_{1}^{20}} & $20$ & $4$ & $0$\\[0.5ex]
\partition{2_{2}^{1},5_{1}^{3}} & \partition{1_{1}^{16},1_{2}^{1}} & $15$ & $3$ & $0$\\[0.5ex]
\partition{2_{2}^{2},5_{1}^{2}} & \partition{1_{1}^{12},1_{2}^{2}} & $10$ & $2$ & $0$\\[0.5ex]
\partition{2_{2}^{3},5_{1}^{1}} & \partition{1_{1}^{8},1_{2}^{3}} & $5$ & $1$ & $0$\\[0.5ex]
\partition{2_{2}^{4}} & \partition{1_{1}^{4},1_{2}^{4}} & $0$ & $0$ & $0$\\[0.5ex]
\partition{2_{1}^{1},2_{3}^{1},5_{1}^{2}} & \partition{1_{1}^{13},1_{3}^{1}} & $12$ & $2$ & $1$\\[0.5ex]
\partition{2_{1}^{1},2_{2}^{1},2_{3}^{1},5_{1}^{1}} & \partition{1_{1}^{9},1_{2}^{1},1_{3}^{1}} & $7$ & $1$ & $1$\\[0.5ex]
\partition{2_{1}^{1},2_{2}^{2},2_{3}^{1}} & \partition{1_{1}^{5},1_{2}^{2},1_{3}^{1}} & $2$ & $0$ & $1$\\[0.5ex]
\partition{2_{1}^{2},2_{4}^{1},5_{1}^{1}} & \partition{1_{1}^{10},1_{4}^{1}} & $9$ & $1$ & $2$\\[0.5ex]
\partition{2_{1}^{2},2_{3}^{2}} & \partition{1_{1}^{6},1_{3}^{2}} & $4$ & $0$ & $2$\\[0.5ex]
\partition{2_{1}^{2},2_{2}^{1},2_{4}^{1}} & \partition{1_{1}^{6},1_{2}^{1},1_{4}^{1}} & $4$ & $0$ & $2$\\[0.5ex]
\partition{2_{1}^{5},5_{1}^{2}} & \partition{1_{1}^{10},10_{1}^{1}} & $10$ & $2$ & $0$\\[0.5ex]
\partition{2_{1}^{5},2_{2}^{1},5_{1}^{1}} & \partition{1_{1}^{6},1_{2}^{1},10_{1}^{1}} & $5$ & $1$ & $0$\\[0.5ex]
\partition{2_{1}^{5},2_{2}^{2}} & \partition{1_{1}^{2},1_{2}^{2},10_{1}^{1}} & $0$ & $0$ & $0$\\[0.5ex]
\partition{2_{1}^{6},2_{3}^{1}} & \partition{1_{1}^{3},1_{3}^{1},10_{1}^{1}} & $2$ & $0$ & $1$\\[0.5ex]
\partition{2_{1}^{10}} & \partition{10_{1}^{2}} & $0$ & $0$ & $0$\\[0.5ex]
 & \partition{1_{4}^{1},10_{1}^{1}} & $-1$ & $-1$ & $2$\\[0.5ex]
 & \partition{1_{4}^{2}} & $-2$ & $-2$ & $4$\\[0.5ex]
 & \partition{1_{1}^{1},1_{2}^{3},1_{3}^{1}} & $-3$ & $-1$ & $1$\\[0.5ex]
 & \partition{1_{1}^{2},1_{2}^{1},1_{3}^{2}} & $-1$ & $-1$ & $2$\\[0.5ex]
 & \partition{1_{1}^{2},1_{2}^{2},1_{4}^{1}} & $-1$ & $-1$ & $2$\\[0.5ex]
 & \partition{1_{1}^{3},1_{3}^{1},1_{4}^{1}} & $1$ & $-1$ & $3$\\[0.5ex]
\hline
\end{tabular}  
\end{table}

Now when we take $z>1$, we get some truly unobvious partition inequalities. For example, if we now take $(K,L,m,n,y,z)=(3,2,4,3,4,3)$, then the domain partitions are constructed from the set 
\[\{12_1,12_2,12_3,9_1,9_2\} = \{12,16,20,9,21\}\]
and the codomain partitions are constructed from the set
\[\{3_1,3_2,3_3,36_1,36_2\} = \{3,7,11,36,48\}.\]
So, in this case Theorem \ref{MainThm} tells us that, for any positive integer $x$, the number of partitions into the parts 9, 12, 16, 20, and 21 is no greater than the number of partitions of $x$ into the parts 3, 7, 11, 36, and 48. This is rather unobvious since, among other things, $9+12+16+20+21 = 78 < 105 = 3+7+11+36+48$ and the largest domain part, 21, is considerably less than even the second largest codomain part, 36. Of course, it is also less obvious since we have neither 1 appearing as a part nor a common factor of all parts in either the domain parts list or the codomain parts list; it is in this way that we may obtain, from Theorem \ref{MainThm}, many examples of irreducible inequalities of the form \eqref{GenericQproductIneq} to answer Yesilyurt's question on the necessity of the factor $(1-q)$. 
We give the complete injection for $(K,L,m,n,y,z)=(3,2,4,3,4,3)$ when $x=60$ in Table \ref{table:example2}. 

\begin{table}[h]
\caption{The injection $\pi_2 \mapsto \pi_1$ for $(K,L,m,n,y,z)=(3,2,4,3,4,3)$, where $|\pi_2|=|\pi_1|=60$. Note that the last eight lines of the table are not part of the injection, i.e.\ those $\pi_1$ have no pre-image $\pi_2$.}\label{table:example2}
\begin{tabular}{c@{\ $\mapsto$\ }cr@{$\ \ {}={}$}r@{$n + {}$}r@{$y$}}
\hline\vphantom{$\dfrac{1}{2}$}
$\pi_2$ & $\pi_1$ & $\mu(\pi_1)$ & $a$ & $b$\\[0.5ex]
\hline\vphantom{$\dfrac{1}{2}$}
\partition{9_{1}^{2},9_{2}^{2}} & \partition{3_{1}^{6},3_{2}^{6}} & $6$ & $2$ & $0$\\[0.5ex]
\partition{12_{3}^{3}} & \partition{3_{1}^{9},3_{3}^{3}} & $0$ & $0$ & $0$\\[0.5ex]
\partition{12_{1}^{1},9_{1}^{3},9_{2}^{1}} & \partition{3_{1}^{13},3_{2}^{3}} & $13$ & $3$ & $1$\\[0.5ex]
\partition{12_{1}^{1},12_{2}^{3}} & \partition{3_{1}^{4},36_{2}^{1}} & $4$ & $0$ & $1$\\[0.5ex]
\partition{12_{1}^{2},9_{1}^{4}} & \partition{3_{1}^{20}} & $20$ & $4$ & $2$\\[0.5ex]
\partition{12_{1}^{2},12_{2}^{1},12_{3}^{1}} & \partition{3_{1}^{14},3_{2}^{1},3_{3}^{1}} & $8$ & $0$ & $2$\\[0.5ex]
\partition{12_{1}^{5}} & \partition{3_{1}^{8},36_{1}^{1}} & $8$ & $0$ & $2$\\[0.5ex]
 & \partition{3_{2}^{7},3_{3}^{1}} & $-6$ & $-2$ & $0$\\[0.5ex]
 & \partition{3_{1}^{1},3_{2}^{3},36_{1}^{1}} & $1$ & $-1$ & $1$\\[0.5ex]
 & \partition{3_{1}^{1},3_{2}^{5},3_{3}^{2}} & $-11$ & $-5$ & $1$\\[0.5ex]
 & \partition{3_{1}^{2},3_{2}^{1},3_{3}^{1},36_{1}^{1}} & $-4$ & $-4$ & $2$\\[0.5ex]
 & \partition{3_{1}^{2},3_{2}^{3},3_{3}^{3}} & $-7$ & $-5$ & $2$\\[0.5ex]
 & \partition{3_{1}^{3},3_{2}^{1},3_{3}^{4}} & $-12$ & $-4$ & $0$\\[0.5ex]
 & \partition{3_{1}^{7},3_{2}^{4},3_{3}^{1}} & $1$ & $-1$ & $1$\\[0.5ex]
 & \partition{3_{1}^{8},3_{2}^{2},3_{3}^{2}} & $-4$ & $-4$ & $2$\\[0.5ex]
\hline
\end{tabular}  
\end{table}

\section{Proofs of the Dual and the Generalization}\label{sec:generalization}

As stated in the introduction, Theorem \ref{MainThm} can really be viewed as a specific instance of Theorem \ref{GenThm}, just with $S=K$ and $T=L$. Our proof of Theorem \ref{GenThm} relies on both Theorem \ref{MainThm}, which was proved in previous sections, and Theorem \ref{DualThm}, in which $K$ and $L$ are switched on the right-hand side of the inequality. So, we begin by proving Theorem \ref{DualThm}.
\begin{proof}[Proof of Theorem \ref{DualThm}]
As before, we can construct an injection to serve as the proof. The main difference between the previous injection and this current injection is how we handle parts with subscripts greater than $L$: we can now handle $(nz)_i$ for $i>L$ just as for $i\leq L$, and we do not need to handle $(yz)_i$ for $i>L$ any more. In fact, this makes the atlas shorter (see Table \ref{table:dual-atlas}).
\begin{table}[ht]
\caption{An atlas for the injection in the proof of Theorem \ref{DualThm}.}\label{table:dual-atlas}
\begin{tabular}{c@{${}\mapsto{}$}cl}\hline
\vphantom{$\dfrac{1}{2}$}\partition{(nz)_i} & \partition{z_i^n} & for $1 \leq i \leq K$\\[0.5ex]
\partition{(yz)_1^{nk+j}} & \partition{z_1^{yj},(nyz)_1^k} & for $0 \leq j < n$, $L>0$\\[0.5ex]
\partition{(yz)_i^{nk+j}} & \partition{z_1^{(y-1)j},z_i^j,(nyz)_i^k} & for $1 < i \leq L$, $0 \leq j < n$\\[0.5ex]
\hline
\end{tabular}
\end{table}

Now suppose that $L=0$; then an injection mapping a partition $\pi_2$ into a partition $\pi_1$ is given by
\begin{equation*}
\num(z_i,\pi_1) = n \cdot \num((nz)_i,\pi_2), \quad 1 \leq i \leq K.
\end{equation*}
If we suppose instead that $L>0$, then an injection mapping a partition $\pi_2$  into a partition $\pi_1$ is given by
\begin{equation*}
\num((nyz)_i,\pi_1) = Q((yz)_i,\pi_2) \quad \text{(if $1\leq i \leq L$)}
\end{equation*}
and
\begin{equation*}
\num(z_i,\pi_1) = 
  \begin{cases}
    n\cdot\num((nz)_i,\pi_2) & \text{if } L < i \leq K,\\
    n\cdot\num((nz)_i,\pi_2) + R((yz)_i,\pi_2) & \text{if } 1 < i \leq L,\\
    n\cdot\num((nz)_1,\pi_2) + y\cdot R((yz)_1,\pi_2) + (y-1)\cdot A &\text{if }i=1,
  \end{cases}
\end{equation*}
where
\begin{equation*}
A :=  \sum_{1 < j \leq L} R((yz)_j,\pi_2).
\end{equation*}

The inverse map when $L=0$ is obvious:
\begin{equation*}
\num((nz)_i,\pi_2) = Q(z_i,\pi_1), \quad 1 \leq i \leq K.
\end{equation*}
For $L>0$, 
if we let $C^\ast(\pi_1)$ be the least nonnegative residue of $\overline{y}\cdot\mu^\ast(\pi_1)$ modulo $n$, let $\overline{y}$ be the multiplicative inverse of $y$ modulo $n$, and define
\begin{equation*}
  \mu^\ast(\pi_1) := \num(z_1,\pi_1) -(y-1)\cdot \sum_{1 < j \leq L} R(z_j,\pi_1),
\end{equation*}
then we have the following for the inverse map:
\begin{equation*}
\num((yz)_i,\pi_2) = 
  \begin{cases}
    n\cdot\num((nyz)_i,\pi_1) + R(z_i,\pi_1) & \text{if } 1 < i \leq L,\\
    n\cdot \num((nyz)_1,\pi_1) + C^\ast(\pi_1) &\text{if }i=1,
  \end{cases}
\end{equation*}
and
\begin{equation*}
\num((nz)_i,\pi_2) =
  \begin{cases}
    Q(z_i,\pi_1) & \text{if } 1 < i \leq K,\\
    \dfrac{\mu^\ast(\pi_1)-y\cdot C^\ast(\pi_1)}{n} & \text{if } i=1.
  \end{cases}
\end{equation*}
\end{proof}

Now, the generalization given by Theorem \ref{GenThm} follows from Theorems \ref{MainThm} and \ref{DualThm}, the fact that $\GEQ$ is transitive, and the simple fact that
\begin{equation*}
\frac{1}{\prod_{r=1}^{s}(1-q^{a_r})} \GEQ \frac{1}{\prod_{r=1}^{s-1}(1-q^{a_r})}
\end{equation*}
regardless of the values of the positive integers $a_1$, \dots, $a_s$, and $s$.

\section{Two Invariants of the Injections}\label{sec:invariant}

It can be useful to consider invariants of a mapping in order to learn more about the mapping. The injections we constructed were designed to, among other things, map a partition to another partition with the same norm. Thus, we consider the norm of a partition to be (the first) invariant under the injections. 

Continuing to employ the subscript notation in section \ref{sec:notation}, we define the following ``partition flattening'' function (not an injection) by its action on parts of a partition.
\begin{definition}
Let $F(\pi)$, where $\pi$ is any partition whose parts are given using the subscript notation in section \ref{sec:notation}, be given by the part-wise action of reducing all subscripts to 1.
\end{definition}
\noindent  In other words, $F$ removes the multiples of $m$ (or $nm$) from the parts. So, for example,
$F\left(\partition{2_1^{},3_2^2,2_3^3}\right) = \partition{2_1^4,3_1^2} = \partition{2^4,3^2}.$

An examination of the atlases (given in Tables \ref{table:atlas} and \ref{table:dual-atlas}) for each of the injections presented yields the following fact:
\[\pi \stackrel{\varphi}{\mapsto} \varphi(\pi) \implies |F(\pi)| = |F(\varphi(\pi))|,\]
where $\varphi$ is either of the injections presented in this manuscript.
So the injections presented not only preserve the norms of the partitions, but they also preserve the norms of the corresponding ``flattened'' partitions. 
Thus, we must have 
\begin{equation}
\tilde P_1(f,x,K,L,n,y,z) \geq \tilde P_2(f,x,K,L,n,y,z),\label{tildes-relationship}
\end{equation}
where $\tilde P_{1}(f,x,K,L,n,y,z)$ and $\tilde P_{2}(f,x,K,L,n,y,z)$ count the same types of partitions as $P_1$ and $P_2$ do in \eqref{firstproduct} and \eqref{secondproduct}, respectively, but with the additional restriction imposed that the partitions must have $|F|$-value equal to $zf$.

Using the $q$-binomial Theorem (Theorem 3.3 in \cite{An2}), we can easily derive 
\begin{align}
\sum_{x=0}^\infty \tilde P_{1}(f,x,K,L,n,y,z) q^{x} = q^{zf}\sum_{\substack{s,t\geq 0,\\ s+nyt=f}} \genqbinom{K-1+s}{s}{q^m} \genqbinom{L-1+t}{t}{q^{mn}} \label{newfirstproduct}\\
\intertext{and}
\sum_{x=0}^\infty \tilde P_{2}(f,x,K,L,n,y,z) q^{x} = q^{zf}\sum_{\substack{s,t\geq 0,\\ sy+nt=f}} \genqbinom{K-1+s}{s}{q^m} \genqbinom{L-1+t}{t}{q^{mn}}. \label{newsecondproduct}
\end{align}
From \eqref{tildes-relationship} we see that \eqref{newfirstproduct} $\GEQ$ \eqref{newsecondproduct}, and so upon dividing by $q^{zf}$ and replacing $q^m$ by $q$ we obtain the following refinement of Theorem \ref{MainThm}.
\begin{theorem} \label{Refinement1}
For any positive integers $n$ and $y$, with $\gcd(n,y)=1$, and integers $K$, $L$, and $f$, with $K \geq L \geq 0$ and $f\geq 0$,
\begin{equation*}
\sum_{\substack{s,t\geq 0\\ s+nyt=f}} \genqbinom{K-1+s}{s}{q} \genqbinom{L-1+t}{t}{q^{n}}\GEQ
\sum_{\substack{s,t\geq 0\\ sy+nt=f}} \genqbinom{K-1+s}{s}{q} \genqbinom{L-1+t}{t}{q^{n}}.
\end{equation*}
\end{theorem}
\noindent
Analogously, we obtain the following refinement of Theorem \ref{DualThm}.
\begin{theorem} \label{Refinement2}
For any positive integers $n$ and $y$, with $\gcd(n,y)=1$, and integers $K$, $L$, and $f$, with $K \geq L \geq 0$ and $f\geq 0$,
\begin{equation*}
\sum_{\substack{s,t\geq 0,\\ s+nyt=f}} \genqbinom{K-1+s}{s}{q} \genqbinom{L-1+t}{t}{q^{n}}\GEQ
\sum_{\substack{s,t\geq 0,\\ sy+nt=f}} \genqbinom{L-1+s}{s}{q} \genqbinom{K-1+t}{t}{q^{n}}.
\end{equation*}
\end{theorem}
\noindent
We also note that Theorem \ref{BGThm} admits a similar type of refinement, as follows.
\begin{theorem} \label{Refinement3}
Suppose $L>0$, $f\geq 0$, and $1<r<m-r$. Then, 
\begin{equation*}
\sum_{\substack{s,t\geq 0,\\ s+(m-1)t=f}} \genqbinom{L-1+s}{s}{q} \genqbinom{L-1+t}{t}{q}\GEQ
\sum_{\substack{s,t\geq 0,\\ sr+(m-r)t=f}} \genqbinom{L-1+s}{s}{q} \genqbinom{L-1+t}{t}{q},
\end{equation*}
provided that $r \doesnotdivide (m-r)$.
\end{theorem}
\noindent
The details of this last refinement, however, will be given elsewhere.

\section{Conclusion}\label{sec:conclusion}

In their original \emph{Lecture hall partitions} paper \cite{BME}, Mireille Bousquet-M\'{e}lou and Kimmo Eriksson proved that the two-variable generating functions for the number of lecture hall partitions with $n$ parts, 
\[\sum_b  X^{b_n + b_{n-2} + \cdots} Y^{b_{n-1} + b_{n-3} + \cdots},\]
where the sum is over all partitions $b=\partition{b_1,\dots,b_n}$ such that 
\[\frac{b_n}{n} \geq \frac{b_{n-1}}{n-1} \geq \cdots \geq b_1 \geq 0,\]
is
\begin{equation}\label{LHPform}
\frac{1}{(X;XY)_n}.  
\end{equation}
We notice that if we take 
$X=q^z$, $Y=q^{m-z}$, and $n=K$ in \eqref{LHPform}, we get the first product in \eqref{MainIneq} with $L=0$; similarly, 
if we take 
$X=q^{yz}$, $Y=q^{m-yz}$, and $n=K$ in \eqref{LHPform}, then we get the second product in \eqref{MainIneq} with $L=0$.

Now suppose instead that we take $K=L$, $n=2$, $m=2(y-1)z$, and $y$ odd in \eqref{MainIneq}; then Theorem \ref{MainThm} implies
\begin{equation}\label{SavageForm}
  \frac{1}{(q^z;q^m)_L(q^{m+2z};q^{2m})_L} \GEQ \frac{1}{(q^{yz};q^m)_L(q^{2z};q^{2m})_L}.
\end{equation}
Building on the work in the subsequent paper \cite{BME2} (by Bousquet-M\'{e}lou and Eriksson),  Sylvie Corteel, Carla Savage, and Andrew Sills established in \cite{CSS} that
\begin{equation}\label{SavageODD}
\sum_b  X^{b_{2L-1} + b_{2L-3} + \cdots + b_1} Y^{b_{2L} + b_{2L-2} + \cdots + b_2} = \frac{1}{(Y;X^2Y^2)_L (X^2Y^4;X^4Y^4)_L},
\end{equation}
where the sum is over all partitions $b=\partition{b_1,\dots,b_{2L}}$ such that 
\[\frac{b_{2L}}{2L} \geq \frac{b_{2L-1}}{2L-1} \geq \cdots \geq b_1 \geq 0\]
and $b_n$ is even whenever $n$ is \emph{odd}; and similarly that
\begin{equation}\label{SavageEVEN}
\sum_b  X^{b_{2L-1} + b_{2L-3} + \cdots + b_1} Y^{b_{2L} + b_{2L-2} + \cdots + b_2} = \frac{1}{(XY^2;X^2Y^2)_L (Y^2;X^4Y^4)_L},
\end{equation}
where now the sum is over all partitions $b=\partition{b_1,\dots,b_{2L}}$ such that 
\[\frac{b_{2L}}{2L} \geq \frac{b_{2L-1}}{2L-1} \geq \cdots \geq b_1 \geq 0\]
and $b_n$ is even whenever $n$ is \emph{even}.
We notice that if we take $X=q^{(y-2)z}$ and $Y=q^z$, then the right-hand sides of \eqref{SavageODD} and \eqref{SavageEVEN} become, respectively, the left-hand and right-hand sides of \eqref{SavageForm}.

If instead we use Theorem \ref{DualThm}, then taking $K=L+1$, $n=2$, $m=2(y-1)z$, and $y$ odd in  \eqref{DualIneq} yields
\begin{equation}\label{DualSavageForm}
  \frac{1}{(q^z;q^m)_{L+1}(q^{m+2z};q^{2m})_L} \GEQ \frac{1}{(q^{yz};q^m)_L(q^{2z};q^{2m})_{L+1}},
\end{equation}
where the left-hand side of \eqref{DualSavageForm} corresponds with
\begin{equation}\label{DualSavageEVEN}
\sum_b  X^{b_{2L} + b_{2L-2} + \cdots + b_2} Y^{b_{2L+1} + b_{2L-1} + \cdots + b_1} = \frac{1}{(Y;X^2Y^2)_{L+1} (X^2Y^4;X^4Y^4)_L},
\end{equation}
the right-hand side of \eqref{DualSavageForm} corresponds with 
\begin{equation}\label{DualSavageODD}
\sum_b  X^{b_{2L} + b_{2L-2} + \cdots + b_2} Y^{b_{2L+1} + b_{2L-1} + \cdots + b_1} = \frac{1}{(XY^2;X^2Y^2)_L (Y^2;X^4Y^4)_{L+1}},
\end{equation}
the sums are over partitions of the form $b=\partition{b_1,\dots,b_{2L+1}}$ with 
\[\frac{b_{2L+1}}{2L+1} \geq \frac{b_{2L}}{2L} \geq \cdots \geq b_1 \geq 0,\]
where for \eqref{DualSavageEVEN} $b_n$ is even whenever $n$ is \emph{even}, and for \eqref{DualSavageODD} $b_n$ is even whenever $n$ is \emph{odd}. In both \eqref{DualSavageEVEN} and \eqref{DualSavageODD} (which were established in \cite{CSS}), we again take $X=q^{(y-2)z}$ and $Y=q^z$ to see the correspondences.

These connections to lecture hall partitions certainly beg the question of wheth\-er or not the more general products in \eqref{MainIneq} have nontrivial partition theoretic interpretations. While we do not have a completely satisfying answer yet, we look forward to learning if they do.

Also, evidence seems to indicate that when $\gcd(n,y)>1$ there may be finitely many exceptions or even infinitely many exceptions (powers of $q$ where the inequality breaks down), depending on the relationship between $n$ and $y$, and to a lesser extent, between $K$ and $L$. When $K<L$ (and $\gcd(n,y)=1$), however, Theorem \ref{MainThm} seems to pathologically fail in the sense that $P_1(x,K,L,m,n,y,z)$ is less than $P_2(x,K,L,m,n,y,z)$ for all $x$ greater than some natural number $X(K,L,m,n,y,z)$. We look forward to discovering the truths hidden behind these observations and sharing them in the future.

Finally, the authors have written some small Maple programs to generate tables like those in section \ref{sec:examples} and would be happy to share them with anyone interested.

\vspace{1em}

\noindent
{\it Acknowledgements}.
We are grateful to the organizers of the Ramanujan 125 conference for their tireless efforts which helped foster an environment eventually leading to the discoveries contained in this manuscript. We thank Carla Savage for bringing \cite{CSS} to our attention as well as for patiently explaining her work on lecture hall partitions to us. We also wish to thank one of the referees of the original manuscript, whose insightful comments led us to the discovery of the refinements discussed in section \ref{sec:invariant}.

\providecommand{\bysame}{\leavevmode\hbox to3em{\hrulefill}\thinspace}

\end{document}